\documentclass[11pt,a4paper,oneside]{amsart}

\usepackage{amsthm}
\usepackage[T1]{fontenc}
\usepackage[utf8]{inputenc}
\usepackage{lmodern}
\usepackage{microtype}
\usepackage{geometry}
\usepackage{amsmath,amssymb,amsthm,mathtools}
\usepackage{hyperref}
\usepackage{fancyhdr}
\usepackage{epigraph}

\usepackage{xcolor}

\numberwithin{equation}{section}

\newcommand{\hide}[1]{}

\newtheorem{theorem}{Theorem}[section]

\newtheorem{lemma}[theorem]{Lemma}

\newtheorem{maintheorem}{Theorem}

\theoremstyle{definition}

\theoremstyle{remark}
\newtheorem{remark}[theorem]{Remark}

\newcommand{\C}{\mathbb C}
\newcommand{\D}{\mathbb D}

\newcommand{\dd}{\,\mathrm d}
\newcommand{\dxdy}{\,\mathrm dx\,\mathrm dy}
\DeclareMathOperator{\supp}{supp}
\DeclareMathOperator{\dist}{dist}
\DeclareMathOperator{\diam}{diam}

\newcommand{\Chat}{\widehat{\mathbb C}}

\DeclareMathOperator{\Pol}{Pol}

\newcommand{\A}{{\mathcal A}}
\newcommand{\Cc}{\widehat{\mathbb C}}
\newcommand{\sm}{\setminus}
\newcommand{\ovl}{\overline}

\renewcommand{\ge}{\geqslant}
\renewcommand{\le}{\leqslant}
\renewcommand{\geq}{\geqslant}
\renewcommand{\leq}{\leqslant}

\renewcommand{\phi}{\varphi}

\title[Bounded-orbit wandering domains do not exist]{Bounded-orbit wandering domains do not exist}

\author[K.~Drach]{Kostiantyn Drach}
\address{Departament de Matem{\`a}tiques i Inform{\`a}tica, Universitat de Barcelona, Gran Via 585, 08007, Barcelona, Spain, and Centre de Recerca Matem{\`a}tica, Edifici C, Carrer de l'Albareda, Bellaterra, Barcelona, Spain}
\email{kostiantyn.drach@ub.edu}

\author[L.~Pardo-Sim\'on]{Leticia Pardo-Sim\'on}
\address{Departament de Matem{\`a}tiques i Inform{\`a}tica, Universitat de Barcelona, Gran Via 585, 08007, Barcelona, Spain, and Centre de Recerca Matem{\`a}tica, Edifici C, Carrer de l'Albareda, Bellaterra, Barcelona, Spain}
\email{lpardosimon@ub.edu}

\author[B.~U\v{c}akar]{Beno U\v{c}akar}
\address{
Faculty of Mathematics and Physics, University of Ljubljana, Jadranska 19, 1000 Ljubljana, Slovenia \newline  \indent
Faculty of Mathematics and Computer Science, University of Barcelona, Gran Via de les Corts Catalanes, 585, 08007 Barcelona, Spain \newline \indent
Institute for Mathematics, Physics and Mechanics, Jadranska 19, 1000 Ljubljana, Slovenia 
}
\email{beno.ucakar@imfm.si}

\date{23 September 2026}

\begin{document}
\maketitle
\vspace{-1.5em}
\begin{abstract}
This paper establishes that transcendental entire functions do not have wandering Fatou components whose forward orbit is bounded. This settles a long-standing question in transcendental dynamics. A  stronger statement is proved using the same method: the orbit of every point in a wandering domain of
a transcendental \emph{meromorphic} function is unbounded. In particular, transcendental meromorphic functions do not admit \emph{orbitally bounded} wandering domains. The proofs rely mostly on uniform estimates for hyperbolic area, following the approach of Ye~\cite{Ye}.
\end{abstract}

\section{Introduction}

Let $f\colon\C\to\Chat$ be a transcendental meromorphic function, including the entire case. We denote its Fatou set by $F(f)$ and its Julia set by $J(f)$. 
A \emph{Fatou component} of $f$ is a connected component of $F(f)$. For such a component $U$, let $U_n$ denote the Fatou component containing $f^n(U)$, $n \ge 0$, with $U_0 = U$. Thus, $f^n(U)\subset U_n$ and $f(U_n)\subset U_{n+1}.$
The Fatou component $U$ is called a \emph{wandering domain} if the components $U_n$ are pairwise distinct. The first result is the following.

\begin{maintheorem}[No bounded-orbit wandering domains]
\label{thm:mainA}
Let $U$ be a wandering Fatou component of a transcendental entire function $f$. Then $\bigcup_{n\geq0}f^n(U)$ is unbounded.
\end{maintheorem}

Note that the statement above does not by itself exclude the possibility that every point of $U$ has a bounded orbit, even though the union of these orbits is unbounded. A wandering domain $U$ is called \emph{orbitally bounded} if $\sup_{n\geq0}|f^n(z)|<\infty$ for every $z\in U$. By normality and since all limit functions are constant (see \cite[Section~4.5]{Bergweiler}), this condition holds for one point of $U$ if and only if it holds
for every point.

The following theorem rules out such domains for all
transcendental meromorphic functions.

\begin{maintheorem}[No orbitally bounded wandering domains]
\label{thm:mainB}
Let $U$ be a wandering Fatou component of a transcendental meromorphic function $f\colon\C\to\Chat$. 
Then
\[\limsup\limits_{n \to \infty} |f^n(z)| = \infty \qquad \text{for every }z \in U.\]
In particular, $f$ has no orbitally bounded wandering domains. 
\end{maintheorem}

\begin{remark}
    No restriction is imposed on the number of poles or on the boundedness or
    connectivity of the components $U_n$. 
    Moreover, note that the above statement is equivalent to the existence of a sequence
    $n_k\to\infty$ such that
    $f^{n_k}\longrightarrow\infty$ locally uniformly on $U$. This answers a question posed by Bergweiler in \cite[Section~4.5, Question~8]{Bergweiler} affirmatively.
\end{remark}

The proof follows the idea of uniform estimates in the hyperbolic metric of \cite{Ye} (see also the next subsection).

The paper is structured as follows. The remainder of the introduction provides some historical context for the bounded-orbit question. In Section~\ref{Sec:Preliminaries}, standard facts from hyperbolic geometry are recalled, and uniform area estimates in the hyperbolic metric are derived in both dynamical and non-dynamical settings. These estimates are used in Section~\ref{Sec:ThmA} to prove Theorem~\ref{thm:mainA}. Although this result follows from Theorem~\ref{thm:mainB}, an independent proof is included to highlight the key ideas. Finally, Theorem~\ref{thm:mainB} is proved in Section~\ref{Sec:ThmB}, using an additional eventual injectivity property on certain hyperbolic balls.

\subsection{Some history of the question}
The question of bounded orbits in wandering domains belongs to the broader problem of understanding which features of polynomial and rational dynamics persist for transcendental functions. Sullivan's no-wandering-domain theorem~\cite{Sul85} settled the problem for rational maps. For transcendental entire functions, Baker~\cite{Bak76} had already exhibited multiply connected escaping wandering
domains, using an infinite-product construction from his earlier work. More generally, every multiply connected Fatou component of a transcendental entire function is bounded and wandering, and the iterates tend locally uniformly to infinity there; see~\cite{Bak84,Bergweiler}. Simply connected escaping examples were also constructed by Herman~\cite{Her84} and Baker~\cite{Bak84}. Eremenko and Lyubich's approximation constructions, announced in~\cite{EL84} and presented in detail in~\cite{EL87}, produced the first \emph{oscillating} wandering domains, whose point orbits accumulate both at finite points and at infinity. These examples have infinitely many finite constant limit functions.

The bounded-orbit question was explicitly raised in~\cite{EL84}, and versions appeared in the problem
lists~\cite{BBH84,BH89}. In the later compilation by Hayman and Lingham~\cite{HL19}, Problem~2.67, attributed to Eremenko, asks whether the set of constant limit functions can be infinite and bounded. Problems~2.77 and~2.87, attributed respectively to Baker, Herman and Kra, and to Herman, Eremenko and Lyubich, ask for a wandering domain whose full forward orbit is bounded. Bergweiler \cite[Section~4.5, Question~8]{Bergweiler} formulated the question for transcendental meromorphic functions: must every wandering domain admit a subsequence of iterates converging locally uniformly to infinity? Theorem~\ref{thm:mainB} answers this question affirmatively.

Several partial nonexistence results were obtained under restrictions on the singular values. Adapting Sullivan's quasiconformal deformation method, Eremenko and Lyubich~\cite{EL84,EL92} and, independently, Goldberg and Keen~\cite{GK86} proved that entire functions with finitely many singular values have no wandering domains. Baker, Kotus, and L\"u~\cite{BKL92} extended this
nonexistence theorem to transcendental meromorphic
functions with finitely many singular values. Their proofs rely on finite-dimensional deformation spaces,
leaving the bounded-orbit question open for general entire functions.

Restrictions on limit functions provided another approach. Bergweiler, Haruta, Kriete, Meier, and Terglane~\cite{BHKMT93} proved that every finite limit function on an entire wandering domain is a non-isolated point of the postsingular set. Zheng~\cite{Zhe03} obtained a corresponding result for
meromorphic functions using the hyperbolic metric. Such restrictions yield no-wandering-domain results when combined with additional information about the singular orbits.

The Eremenko--Lyubich class $\mathcal B$, consisting of transcendental entire functions with bounded singular sets, was a natural setting for the bounded-orbit question. Eremenko and Lyubich~\cite[Theorem~1]{EL92} proved that Fatou points of functions in this class cannot escape to
infinity. Nevertheless, Bishop~\cite{Bis15} constructed oscillating wandering domains in $\mathcal B$ by quasiconformal folding. Fagella, Jarque, and Lazebnik~\cite{FJL19} subsequently constructed oscillating wandering domains in $\mathcal B$ with univalent dynamics throughout the wandering orbit.
Using quasiregular interpolation, Mart\'i-Pete and Shishikura~\cite{MS20} obtained oscillating examples for functions in $\mathcal B$ of finite order, including order $1/2$.  For real functions in $\mathcal B$ whose singular values are all real, the rigidity results of Rempe and van Strien~\cite{RvS15} imply the absence of orbitally bounded wandering domains; see also
\cite[Proposition~5.3]{MRG13}. Under an additional sector condition, Mihaljevi\'c-Brandt and Rempe~\cite{MRG13} excluded all wandering domains using estimates for the hyperbolic metric. For further results on absence of wandering domains for different classes of meromorphic functions, see \cite{BM,BFJK}. For a recent account of the approximation constructions and their
development from the work of Eremenko and Lyubich, see the survey \cite{FP26}.

The holomorphic hypothesis is also significant. Nicks \cite{Nic13} constructed a quasiregular map of the plane with an essential singularity at infinity and a wandering orbit contained in a bounded region. Hence an analogous exclusion does not hold for general quasiregular maps. 

The method used here is motivated by the recent work of Ye \cite{Ye}. Using hyperbolic area, Ye gave a proof of Sullivan's theorem and of the finite-singular-set theorem without quasiconformal deformation theory. His argument also applies to entire functions with compact singular set
whose accumulation points lie in the Fatou set. We develop this approach to rule out bounded wandering orbits without assumptions on the singular set.

\subsection*{AI Disclaimer and Acknowledgments}

This paper was obtained with the use of AI (ChatGPT-6 Astra). The authors checked the mathematics and take full responsibility (and, if necessary, the blame) for the content.

Confirming a belief shared by many experts in the field, once an alternative proof to Sullivan's No Wandering Domains Theorem was found by Ye \cite{Ye}, also with the use of AI, it became apparent that the same method could be applied to one of the main open questions in transcendental dynamics: the non-existence of wandering domains with bounded orbits. The proof, and the route to it, turned out to be surprisingly simple; only a few targeted prompts were needed. We were not the only ones to explore this possibility: parallel and independent work was carried out by Prochorov, Rempe, and Waterman~\cite{PRW}.

The mathematical content of this paper, as well as the fun of writing it, was shared ``live'' with the members of the Holomorphic Dynamics group at the University of Barcelona. We thank them all: Jordi Canela, Núria Fagella, Xavier Jarque, Anna Jové, and especially Gustavo R.\ Ferreira. Perhaps this is one of the new ways in which mathematical research will be done in the age of AI. 
A result might no longer be attributed to individuals 
but rather to the community that cares to understand it, refine it, and expand upon it. The authors do not doubt that this work will prompt further research, with or without the use of AI.

The authors were partially supported by Agencia Estatal de Investigaci\'on grants PID2023-147252NB-I00, CNS2025-166633, by the Severo Ochoa and Mar\'ia de Maeztu Program for Centers and Units of Excellence in R\&D (CEX2020-001084-M), and by the research program P1-0291 from ARIS, Republic of Slovenia. LP is a Serra Húnter fellow.

\section{Preliminaries}
\label{Sec:Preliminaries}

\subsection{Preliminaries from hyperbolic geometry}
\label{SSec:PrelimHyp}

In this subsection, we collect some standard facts from hyperbolic geometry in complex dimension one; see, e.g., \cite{BM}. 

For a hyperbolic domain $G \subset \C$, let $\rho_G$ be the complete Poincar\'e density of the hyperbolic metric of curvature $\kappa=-4$, with the standard normalization $\rho_{\D}(z)={1}/({1-|z|^2})$, where $\D \subset \C$ is the unit disk. We denote by $d_G$ the corresponding hyperbolic distance function in $G$. The curvature equation gives
\begin{equation}\label{eq:curvature}
\Delta\log\rho_G=4\rho_G^2,
\end{equation}
where $\Delta=\partial_x^2+\partial_y^2$ is the Laplacian and $z = x+iy$.

For a holomorphic map $h\colon H\to G$ between two hyperbolic domains, the \emph{pullback metric} on $H$ has density $(h^*\rho_G)(z) :=\rho_G(h(z))|h'(z)|$. By the Schwarz--Pick lemma, we have the comparison
\[
h^*\rho_G\leq\rho_H, \qquad\text{and}
\qquad
h^*\rho_G=\rho_H\quad\text{if and only if $h$ is a covering}.
\]
In particular, if $G\subset H$, then $\rho_G\geq\rho_H$ on $G$.

For a nonnegative density $\rho$ and a Borel set $A$ in its domain, put 
\[
\A_\rho(A)=\int_A\rho(z)^2\dxdy 
\]
for the area with density $\rho$. (When no confusion can arise, we write $\int_A \rho^2\dxdy$ for $\A_\rho(A)$, omitting the variable in the density.) If $h \colon H \to G$ is injective, then for $A \subset H$ we have the change-of-variables formula
\begin{equation*}
\label{eq:change-variables}
\A_{\rho_G}(h(A))=\int_A(h^*\rho_G)^2\dxdy.
\end{equation*}
If $h \colon H \to G$ is finite-to-one, and $q \ge 0$ is a Borel function, then the corresponding formula becomes
\begin{equation*}
\label{eq:multiplicity}
\int_A q(h(z))|h'(z)|^2\dxdy
=\int_{h(A)}q(w)N_A(w)\,\mathrm d(\Re w)\,\mathrm d(\Im w),
\end{equation*}
where $N_A(w) = \#|\{z \in A : h(z) = w\}|$ is the number of preimages in $A$. Critical values have zero planar measure, so multiplicities at those values do not affect the integral.

For a finite set $P\subset\C$ with cardinality $\#|P| \geq2$, applying the Gauss--Bonnet formula to the sphere $\Cc$ punctured at $P\cup\{\infty\}$, and using the normalization of the curvature $\kappa = -4$, yields
\[
-4\cdot \A_{\rho_{\C\setminus P}}(\C\setminus P) = \int_{\C \sm P} \kappa \, \rho_{\C \sm P}^2 \dxdy = 2\pi\chi(\C \sm P) = 2\pi(1 - \#|P|).
\]
Hence, the total hyperbolic area 
\begin{equation}
    \label{eq:finite-area}
    \A_{\rho_{\C \sm P}} (\C \sm P) = \frac{\pi}{2}(\#|P|-1)   
\end{equation}
is finite.

\subsection{Comparison of hyperbolic metrics with respect to domains}
\label{SSec:HyperDom}

The following lemma will play a key role in the finite exhaustion argument and the area estimates in the proofs of the main results. This lemma and its corollary (Lemma~\ref{lem:punctures}) are the finite-puncture facts used in \cite[Lemmas~2.2--2.3]{Ye}.

For a set $A \subset \C$, let $\ovl A$ denote the closure in $\C$. 

\begin{lemma}[Recovery of the intrinsic metric]
\label{lem:recovery}
Let $E\subset\C$ be a closed set. Suppose that $P_j\subset E$ are increasing finite sets, $\#|P_1|\geq2$, that exhaust $E$:
\[
\overline{\bigcup_{j\geq1}P_j}=E.
\]
Then on every connected component $G$ of $\C\setminus E$, we have the monotone convergence
\[
\rho_{\C \sm P_j}\big|_G \uparrow \rho_G \quad \text{locally uniformly in $G$.}
\] 
In particular, for every Borel set $A\subset G$,
\[
\lim_{j \to \infty} \A_{\rho_{\C \sm P_j}}(A) = \A_{\rho_G}(A).
\]
\end{lemma}

\begin{proof}
The proof is a corollary of the pointed Carath\'eodory convergence of domains $(\C \sm P_j, z_0) \to (G,z_0)$ for any choice $z_0 \in G$; see, e.g., \cite[Proposition 4.4]{McM}.
\end{proof}

The following lemma is a corollary to the lemma above.

\begin{lemma}[Cost of finitely many punctures]
\label{lem:punctures}
Let $G\subset\C$ be a hyperbolic domain, and let $E\subset G$ be a finite set. Then
\begin{equation}\label{eq:puncture-cost}
0\leq
\int_{G\setminus E}
\left(\rho_{G\setminus E}^2-\rho_G^2\right)\dxdy
\leq\frac{\pi}{2}\#|E|.
\end{equation}
\end{lemma}

\begin{proof}
Choose increasing finite sets $P_j\subset\C\setminus G$, with $\#|P_1| \ge 2$, that exhaust $\C\setminus G$, i.e.,
\[
\ovl{\bigcup_{j \ge 1} P_j} = \C \sm G.
\]
By \eqref{eq:finite-area},
\begin{equation*}
\begin{aligned}
&\A_{\rho_{\C \sm P_j}}(\C \sm P_j) = \frac{\pi}{2}(\#|P_j|-1), \\
&\A_{\rho_{\C \sm (P_j \cup E)}}(\C \sm (P_j \cup E)) = \frac{\pi}{2}(\#|P_j \cup E|-1) = \frac{\pi}{2}(\#|P_j| + \#|E|-1),   
\end{aligned}
\end{equation*}
where the last equality follows since the sets $P_j$ and $E$ are disjoint. Hence
\[
\int_{\C \sm (P_j \cup E)} \rho_{\C \sm (P_j \cup E)}^2\dxdy - \int_{\C \sm P_j} \rho_{\C \sm P_j}^2\dxdy  = \frac{\pi}{2} \#|E|.
\]
Furthermore, since $E$ is finite (and hence has measure zero), we obtain
\[
\int_{\C \sm P_j} \rho_{\C \sm P_j}^2\dxdy = \int_{\C \sm (P_j \cup E)} \rho_{\C \sm P_j}^2\dxdy.
\]
Combining the last two estimates, we get
\[
\int_{\C \sm (P_j \cup E)} \left(\rho_{\C \sm (P_j \cup E)}^2 - \rho_{\C \sm P_j}^2\right)\dxdy = \frac{\pi}{2} \#|E|.
\]
The integrand is nonnegative by monotonicity. Lemma~\ref{lem:recovery} gives the convergence of $\rho_{\C\sm(P_j \cup E)}$ and $\rho_{\C\sm P_j}$ to the corresponding intrinsic hyperbolic densities $\rho_{G \sm E}$ and $\rho_G$ on compacts in $G \sm E$. Restricting the integrals to $G\setminus E$ and applying Fatou's lemma proves \eqref{eq:puncture-cost}.
\end{proof}

\begin{lemma}[Comparison of densities with the same punctures]
\label{lem:localization}
Fix a finite set $Q\subset\C$, with $\#|Q|\geq2$, an open Euclidean disk $D \subset \C$, and a compact set $K\subset D$. Then there exists a constant $C=C(Q,D,K)<\infty$ such that, for every finite set $P\supset Q$,
\[
\int_{K\setminus P}
\left(\rho_{D\setminus P}^2-\rho_{\C\setminus P}^2\right)\dxdy
\leq C.
\]
\end{lemma}

Note that we do not assume that $K$ and $P$ are disjoint. The integrand in Lemma~\ref {lem:localization} measures pointwise the increase in the squared hyperbolic density caused by restricting the ambient domain from $\mathbb C\setminus P$ to $D\setminus P$. Thus, the integral over $K$ measures the total accumulated increase of hyperbolic area density on $K$; the lemma states that this excess is uniformly bounded, independently of the additional punctures in $P\supset Q$.

\begin{proof}[Proof of Lemma~\ref{lem:localization}]
Set $X_P :=\C\setminus P$ and define the function
\[
u_P :=\log\frac{\rho_{D\setminus P}}{\rho_{X_P}}
\quad\text{on }D\setminus P.
\]
Domain monotonicity and \eqref{eq:curvature} give
\[
u_P\geq0,
\qquad
\Delta u_P=4\left(\rho_{D\setminus P}^2-\rho_{X_P}^2\right)\geq0,
\]
so $u_P$ is a non-negative subharmonic function on $D \sm P$.

Fix a compact set $T \subset D$. Choose a closed circular annulus inside $D$, separating $T$ from $\partial D$, which avoids the finite set $Q$. The positive continuous function $\rho_{\C\setminus Q}$ has a positive minimum on this annulus. Every path from $T\setminus P$ to $X_P\setminus D$ crosses the annulus. Since $\rho_{X_P} = \rho_{\C \sm P} \geq\rho_{\C\setminus Q}$, there is $\varepsilon_T>0$, independent of $P$, such that
\begin{equation}\label{eq:distance-bound}
d_{X_P}(z,X_P\setminus D)\geq\varepsilon_T
\qquad \forall z\in T\setminus P,
\end{equation}
where $d_{X_P}$ is the hyperbolic distance in $X_P = \C \sm P$.

Let $\pi\colon\D\to X_P$ be a universal cover with $\pi(0)=z\in T\setminus P$, and put $a_T :=\tanh\varepsilon_T$; note that we can choose $a_T < 1$. By covering invariance, every radial segment in $a_T\D$ has image of hyperbolic length less than $\varepsilon_T$. By~\eqref{eq:distance-bound},
\[
\pi(a_T\D)\subset D\setminus P.
\]
By the Schwarz--Pick Lemma applied to the mapping $\pi \colon a_T \mathbb D \to D \sm P$ at $0$, we obtain
\[
\rho_{D \sm P}(z) |\pi'(0)| \le \rho_{a_T \mathbb D}(0) = \frac{\rho_{\mathbb D}(0)}{a_T} = \coth \varepsilon_T.
\]
Using the same idea for the starting covering $\pi\colon\D\to X_P$, we get $\rho_{X_P}(z)|\pi'(0)|=1$. Therefore,
\begin{equation}\label{eq:uniform-log-bound}
0\leq u_P(z)\leq\log\coth\varepsilon_T \qquad \forall z\in T\setminus P.
\end{equation}
Thus $u_P$ is locally bounded near each puncture $P\cap D$. A subharmonic function locally bounded above near an isolated puncture extends subharmonically across that puncture. Thus $u_P$ extends to a nonnegative subharmonic function on $D$. By \eqref{eq:uniform-log-bound}, these extensions are uniformly bounded on each compact subset of $D$.

Since $u_P$ extends to a non-negative subharmonic function on $D$, we may mollify it in the usual way. Let $\eta\in C_c^\infty(\mathbb C)$ be a non-negative standard mollifier, supported in the unit disk and normalized by $\int_{\mathbb C}\eta\,\dxdy=1$. For $\delta>0$, let
\[
u_{P,\delta}(z)
:=
\int_{\mathbb C}u_P(z-\zeta)\eta_\delta(\zeta)\,\dxdy(\zeta),
\]
where $\eta_\delta(\zeta)=\delta^{-2}\eta(\zeta/\delta)$ is a re-scaled mollifier. In this way, $u_{P,\delta}$ is well defined and smooth on
\[
D_\delta:=\{z\in D:\dist(z,\partial D)>\delta\}
\]
(where $\dist$ stands for the standard Euclidean distance). Moreover, since $u_P$ is subharmonic,
\[
\Delta u_{P,\delta}
=
(\Delta u_P)*\eta_\delta\ge 0.
\]

Choose $\phi\in C_c^\infty(D)$ such that $0\le \phi\le 1$ and $\phi\equiv 1$ on a neighbourhood of $K$. Taking $\delta>0$ sufficiently small, we may assume that $\supp\phi\subset D_\delta$. Since $u_P$ is smooth
on $D\setminus P$, we have
\[
\Delta u_{P,\delta}
\longrightarrow
\Delta u_P
=
4\left(\rho_{D\setminus P}^2-\rho_{\mathbb C\setminus P}^2\right)
\]
locally uniformly on $D\setminus P$. As $\Delta u_{P,\delta}\ge0$, Fatou's
lemma gives
\[
4\int_{K\setminus P}
\left(\rho_{D\setminus P}^2-\rho_{\mathbb C\setminus P}^2\right)\,\dxdy
\le
\liminf_{\delta\to0}
\int_D \phi\,\Delta u_{P,\delta}\,\dxdy .
\]

Since $\phi$ has compact support in $D$, integration by parts yields
\[
\int_D \phi\,\Delta u_{P,\delta}\,\dxdy = \int_D u_{P,\delta}\,\Delta\phi\,\dxdy .
\]
By the uniform bound obtained above, there exists a constant $M>0$,
depending only on $Q$, $D$, and $\supp\phi$ (and hence on $K$), such that $0\le u_P\le M$ on a neighbourhood of $\supp\phi$, uniformly in $P\supset Q$. Hence, for
all sufficiently small $\delta$, we have the same bound $0\le u_{P,\delta}\le M$ on $\supp\phi$. Therefore,
\[
\int_D u_{P,\delta}\,\Delta\phi\,\dxdy
\le
M\int_D |\Delta\phi|\,\dxdy .
\]
Consequently,
\[
\int_{K\setminus P}
\left(\rho_{D\setminus P}^2-\rho_{\mathbb C\setminus P}^2\right)\,\dxdy
\le
\frac{M}{4}\int_D |\Delta\phi|\,\dxdy .
\]
The right-hand side is independent of $P$, which proves the required
uniform bound.
\end{proof}

\subsection{Comparison of hyperbolic metrics under pullbacks}
\label{SSec:HyperMap}

In this subsection, we apply the estimates of Subsection~\ref{SSec:HyperDom} to the dynamical setting, i.e., when the hyperbolic metrics on certain domains are related by the dynamics of a meromorphic (or an entire) map. 

In what follows, we denote by $\Pol(f)$ the set of poles of a transcendental meromorphic map $f \colon \C \to \Cc$. If $f$ is entire, then $\Pol(f) = \emptyset$.

The first preliminary lemma is trivial, and we provide its proof for completeness. 

\begin{lemma}[Local charts]
\label{lem:local-charts}
Let $f\colon\C\to\Chat$ be a nonconstant meromorphic map. Every point
$a\in\C\setminus\Pol(f)$ has arbitrarily small neighborhoods
$A\Subset\C\setminus\Pol(f)$ such that
\[
 f\colon A\longrightarrow D
\]
is a proper map of finite degree, where $D$ is a Euclidean disk. The restriction
has no critical value or has the single critical value $f(a)$. Consequently, every compact set $K\subset\C\setminus\Pol(f)$ admits finitely many such charts $A_\ell$ and compact sets $K_\ell\Subset A_\ell$ such that $\bigcup K_\ell \supset K$.
\end{lemma}

\begin{proof}
Fix $a\in\C\setminus\Pol(f)$. Since $f(a)\in\C$, the local normal form for a nonconstant holomorphic map gives a neighborhood $U$ of $a$, a conformal coordinate $\zeta\colon U\to \zeta(U)$ with $\zeta(a)=0$, and an integer $m\ge 1$ such that
\[
f(z)=f(a)+\zeta(z)^m,\qquad z\in U.
\]
Choosing $r>0$ arbitrarily small with $\overline{D(0,r)}\subset\zeta(U)$, set
\[
A:=\zeta^{-1}(D(0,r)), \qquad D:=D(f(a),r^m).
\]
Then $A\Subset\C\setminus\Pol(f)$, and, in the coordinate $\zeta$, the restriction $f\colon A\to D$ is conjugate to
\[
w\longmapsto w^m\colon D(0,r)\to D(0,r^m).
\]
Hence it is a proper map of degree $m$. If $m=1$, it has no critical points; if $m\ge2$, its unique critical point is $a$, and its unique critical value is $f(a)$. Since $r$ may be chosen arbitrarily small, such charts form a neighborhood basis at $a$.

For the final assertion, choose for each $a\in K$ such a chart $A_a$ and an open neighborhood $U_a$ of $a$ with ${U_a}\Subset A_a$. By compactness, finitely many $U_{a_\ell}$ cover $K$. Setting
\[
A_\ell:=A_{a_\ell},
\qquad
K_\ell:=\overline{U_{a_\ell}},
\]
gives $K_\ell\subset A_\ell$ and $\bigcup_\ell K_\ell\supset K$.
\end{proof}

For $s\in\mathbb R$, write $[s]_+=\max\{s,0\}$ for the positive part of $s$. The following lemma is the main claim in this subsection.

\begin{lemma}[Uniform local metric deficit under pullbacks]
\label{lem:deficit}
Let $f\colon\C\to\Cc$ be a nonconstant meromorphic function. Fix a finite set $Q\subset\C$, with $\#|Q|\geq2$, and a compact set $K\subset\C\sm\Pol(f)$. Then there exists a positive constant $C = C(f, Q, K)<\infty$ such that, for every finite, forward-invariant set $P\supset Q$, it follows that
\begin{equation}\label{eq:deficit}
\int_K\left[\rho_{\C \sm P}^2-(f^*\rho_{\C \sm P})^2\right]_+\dxdy\leq C.
\end{equation}
The discrete exceptional sets on which the displayed densities are undefined are omitted from the integral.
\end{lemma}

\begin{proof}
Choose finitely many proper charts
\[
f\colon A_\ell\to D_\ell
\]
and compact sets $K_\ell\Subset A_\ell$ covering $K$, as in
Lemma~\ref{lem:local-charts}. Let $V_\ell\subset D_\ell$ denote the set of
critical values of $f|_{A_\ell}$, and put
\[
Z_{\ell,P}:=D_\ell\setminus(P\cup V_\ell).
\]
Then
\[
f\colon A_\ell\setminus f^{-1}(P\cup V_\ell)\longrightarrow Z_{\ell,P}
\]
is a covering. Since $f(P)\subset P$, we have
\[
A_\ell\setminus f^{-1}(P\cup V_\ell)\subset\C\setminus P.
\]
Hence, by domain monotonicity and covering invariance of the hyperbolic
metric,
\[
\rho_{\C\setminus P}
\le
\rho_{A_\ell\setminus f^{-1}(P\cup V_\ell)}
=
f^*\rho_{Z_{\ell,P}}.
\]
Thus, we have
\[
\rho_{\C\setminus P}(z)^2 \le |f'(z)|^2\rho_{Z_{\ell,P}}(f(z))^2 \qquad \text{ for every } z\in K_\ell\setminus f^{-1}(P\cup V_\ell).
\]
Since
\[
\rho_{Z_{\ell,P}}\ge \rho_{\C\setminus P},
\]
it follows that, almost everywhere on $K_\ell$,
\begin{equation}\label{eq:local-deficit-bound}
\left[\rho_{\C\setminus P}^2-(f^*\rho_{\C\setminus P})^2\right]_+
\le
|f'|^2
\left(\rho_{Z_{\ell,P}}^2-\rho_{\C\setminus P}^2\right)\circ f .
\end{equation}

We decompose the difference on the target as
\[
\rho_{Z_{\ell,P}}^2-\rho_{\C\setminus P}^2
=
\left(\rho_{D_\ell\setminus P}^2-\rho_{\C\setminus P}^2\right)
+
\left(\rho_{Z_{\ell,P}}^2-\rho_{D_\ell\setminus P}^2\right).
\]
Both terms are nonnegative. By Lemma~\ref{lem:localization}, the integral
of the first over $f(K_\ell)\Subset D_\ell$ is bounded by a constant
$C_\ell$ independent of $P$. By Lemma~\ref{lem:punctures}, applied in
$D_\ell\setminus P$ to the finite set $V_\ell\setminus P$, the second has
total integral at most
\[
\frac{\pi}{2}\#|V_\ell\setminus P|
\le \frac{\pi}{2}\#|V_\ell|.
\]

Let $d_\ell:=\deg(f\colon A_\ell\to D_\ell)$. Applying the area formula to \eqref{eq:local-deficit-bound}, and using that the multiplicity of $f|_{K_\ell}$ is at most $d_\ell$, gives
\[
\int_{K_\ell}
\left[\rho_{\C\setminus P}^2-(f^*\rho_{\C\setminus P})^2\right]_+\dxdy
\le
d_\ell\left(C_\ell+\frac{\pi}{2}\#|V_\ell|\right).
\]
Summing over the finitely many sets $K_\ell$ yields \eqref{eq:deficit},
with a constant independent of $P$.
\end{proof}

\subsection{Dynamical finite-puncture approximation}
\label{SSec:DA}

In this subsection, we explain the main dynamical approximation by domains in the complement of an increasing and exhausting set of points. The same approximation is used in \cite{Ye}.

Let $f \colon \C \to \Cc$ be a transcendental meromorphic map. By \cite[Theorem~4]{Bergweiler}, there are increasing finite unions $P_j$ of repelling periodic orbits of $f$ such that
\[
\#|P_1|\geq2,\qquad f(P_j)=P_j,\qquad
\overline{\bigcup_{j\geq1}P_j}=J(f).
\]
We fix one such sequence and put
\begin{equation}\label{eq:rho-j}
\rho_j :=\rho_{\C\setminus P_j}.
\end{equation}
Then by \eqref{eq:finite-area}, 
\begin{equation}
\label{eq:finite-area-j}
\A_{\rho_j}(\C\setminus P_j)=\frac\pi2(\#|P_j|-1)<\infty,
\end{equation}
and Lemma~\ref{lem:recovery} gives, on every Fatou component $V$ of $f$, we have the locally uniform monotone convergence
\begin{equation}\label{eq:fatou-recovery}
\rho_j|_V \uparrow \rho_V.
\end{equation}
These statements do not require $V$ to be bounded or simply connected.

\section{Proof of Main Theorem~\ref{thm:mainA}}
\label{Sec:ThmA}

In this section, $f \colon \C \to \C$ is a transcendental entire function. The following lemma is well-known.

\begin{lemma}[Images of bounded components]\label{lem:proper-component}
Let $U$ be a bounded Fatou component, and let $V$ be the Fatou component containing $f(U)$. Then $V$ is bounded and
\[
f\colon U\longrightarrow V
\]
is proper, surjective, and of finite degree. In particular, $f(U)=V$. \qed
\end{lemma}

Finally, recall the following fact due to Baker, applicable to any wandering domain (bounded or not):

\begin{lemma}[{\cite[Theorem 3.1]{Bak84}}]\label{lem:simple-connectivity}
Let $V$ be a Fatou component of $f$ containing a point with a bounded orbit. Then $V$ is simply connected.  \qed
\end{lemma}

\begin{proof}[Proof of Theorem~\ref{thm:mainA}]
Let $U$ be a wandering domain of $f$, and suppose that the orbit of $U$ is contained in a compact disk $K$. In particular, $U$ is bounded. Applying Lemma~\ref{lem:proper-component} inductively gives
\begin{equation}\label{eq:bounded-union-components}
U_n :=f^n(U)\subset K,
\qquad f\colon U_n\longrightarrow U_{n+1}
\quad\text{proper and onto}.
\end{equation}
All point orbits in the domains $U_n$ stay in $K$, so every $U_n$ is simply connected by Lemma~\ref{lem:simple-connectivity}.

The compact set $K$ contains only finitely many zeros of $f'$. Since the domains $U_n$ are pairwise disjoint, there is $N$ such that $f'$ has no zero in $U_n$ for $n\geq N$. For these $n$, the restriction in \eqref{eq:bounded-union-components} is a proper unramified covering of a simply connected domain. Hence it is a conformal isomorphism.

Use the densities $\rho_j$ from \eqref{eq:rho-j}. Lemma~\ref{lem:deficit}, with $Q=P_1$, gives a constant $C<\infty$ (independent of $j$) such that
\begin{equation}
\label{eq:uniform-deficit-K}
\int_K\left[\rho_j^2-(f^*\rho_j)^2\right]_+\dxdy\leq C
\qquad(j\geq1).
\end{equation}
Fix $j$ and put $a_{j,n} := \A_{\rho_j}(U_n)$. These quantities are finite. For $n\geq N$, the change-of-variables formula for the conformal isomorphism $f\colon U_n\to U_{n+1}$ gives
\begin{align*}
a_{j,n}-a_{j,n+1}
&=\int_{U_n}\left(\rho_j^2-(f^*\rho_j)^2\right)\dxdy\\
&\leq\int_{U_n}\left[\rho_j^2-(f^*\rho_j)^2\right]_+\dxdy.
\end{align*}
Summing from $n=N$ to some $n=L > N$, and using disjointness and containment in $K$, yields
\begin{equation}\label{eq:main-telescoping}
a_{j,N}-a_{j,L+1}\leq C.
\end{equation}
On the other hand, for any fixed $j$, we have
\[
\sum_{n\geq N}a_{j,n} \leq \A_{\rho_j}(\C\setminus P_j)<\infty,
\]
so $a_{j,L+1}\to0$ as $L \to \infty$. Sending $L$ to infinity in \eqref{eq:main-telescoping} gives $\A_{\rho_j}(U_N)\leq C$.

Now let $j\to\infty$. By \eqref{eq:fatou-recovery} and monotone convergence,
\[
\A_{\rho_{U_N}}(U_N)\leq C.
\]
But $U_N$ is simply connected and hyperbolic. Thus $\A_{\rho_{U_N}}(U_N)=\infty$, a contradiction.
\end{proof}

\section{Proof of Main Theorem~\ref{thm:mainB}}
\label{Sec:ThmB}

Throughout this section, let $f \colon \C \to \Cc$ be a transcendental meromorphic function, $U$ a wandering domain, and let $z_0\in U$ have a bounded orbit. We write
\begin{equation}\label{eq:bounded-orbit}
 z_n=f^n(z_0),\qquad |z_n|\leq M\quad(n\geq0),\qquad
 S=\overline{\{z_n:n\geq0\}}.
\end{equation}
The compact set $S$ contains no pole. Indeed, no orbit point is a pole;
and if a subsequence converged to a pole $p$, its successors would tend
to infinity, contradicting \eqref{eq:bounded-orbit}.

Because the set of poles is closed and discrete, we can choose $\eta>0$ such that $f$
is holomorphic on a neighborhood of
$\{z:\dist(z,S)\leq2\eta\}$. We fix
\begin{equation}\label{eq:pole-free-K}
 K=\{z\in\C:\dist(z,S)\leq\eta\}
\end{equation}
and note that each disk $D(z_n,2\eta)$ is pole-free. 

For $t>0$, we define the ambient hyperbolic balls
\[
 B_n(t)=\{w\in U_n:d_{U_n}(w,z_n)<t\}
\]
and choose universal covers
\[
 \pi_n\colon\D\longrightarrow U_n,\qquad \pi_n(0)=z_n.
\]
Note, that these are not necessarily injective. Covering invariance and path lifting show that, for $r_t=\tanh t$,
\begin{equation}\label{eq:ball-uniformization}
 B_n(t)=\pi_n(r_t\D).
\end{equation}
In particular, $B_n(t)$ is connected and relatively compact in $U_n$.

\begin{lemma}[Shrinking and forward inclusion]\label{lem:shrinking}
For every fixed $t>0$,
\[
 \sup_{w\in B_n(t)}|w-z_n|\longrightarrow0.
\]
Consequently, $\diam B_n(t)\to0$ and $B_n(t)\subset K$ for all sufficiently
large $n$. Moreover,
\begin{equation}\label{eq:ball-inclusion}
 f(B_n(t))\subset B_{n+1}(t)\qquad(n\geq0).
\end{equation}
\end{lemma}

\begin{proof}
Fix two distinct finite points of $J(f)$. Every map $\pi_n$ omits them as well as
infinity, so the family $(\pi_n)$ is normal by the Montel's theorem. 
Given any subsequence of $(\pi_n)$, pass to
a further subsequence such that $z_n\to a\in S$ and $\pi_n$ converges
spherically locally uniformly to a limit $\pi$. Since $\pi(0)=a$, the limit function is not infinity.

We claim that $\pi$ is a constant function. If not, we can choose $\zeta_0\in\D$ with
$\pi'(\zeta_0)\neq0$, set $w=\pi(\zeta_0)$, and find a sufficiently small
closed disk centered at $\zeta_0$ on whose boundary $\pi-w$ is nonzero.
Rouch\'e's theorem implies that all sufficiently late $\pi_n$ in this
subsequence attain $w$ in that disk. 
It follows, that $w$ lies in infinitely many distinct Fatou components $U_n$, which is a contradiction.  

Thus every such limit is the constant $a$. It follows, by uniform convergence, that for each $r<1$,
\[
 \sup_{|\zeta|\leq r}|\pi_n(\zeta)-z_n|\longrightarrow0.
\]
The result now follows by using \eqref{eq:ball-uniformization}.
To obtain the forward inclusion, we apply the Schwarz--Pick lemma for the holomorphic map $f\colon U_n\to U_{n+1}$.
\end{proof}

\begin{lemma}[Local filling inside pole-free disks]\label{lem:local-filling}
Fix $t>0$. For all sufficiently large $n$, every Jordan curve
$\gamma\subset B_n(t)$ has its bounded interior contained in $U_n$.
\end{lemma}

\begin{proof}
By Lemma~\ref{lem:shrinking}, we can choose $N = N(t)$ large enough such that
\[
 B_m(t)\subset D(z_m,\eta)\qquad(m\geq N).
\]
Fix $n\geq N$. Let $\gamma\subset B_n(t)$ be a Jordan curve and let $\Omega$ denote its bounded interior. 
Since $\gamma\Subset D(z_n,\eta)$, the maximum modulus principle applied to the identity
map gives
\[
 \overline\Omega\subset D(z_n,\eta).
\]

We prove by induction that, for every $k\geq0$,
\begin{equation}\label{eq:filled-orbit}
 \begin{gathered}
 f^k\text{ is holomorphic on a neighborhood of }\overline\Omega,\\
 f^k(\overline\Omega)\subset D(z_{n+k},\eta).
 \end{gathered}
\end{equation}
The case $k=0$ follows by our choice of $\Omega$, 
so suppose that our assertion holds for $k$.
Then $f^k(\overline\Omega)$ lies in the pole-free disk $D(z_{n+k},\eta)$ and the composition 
$f^{k+1}=f\circ f^k$ is holomorphic on a neighborhood of $\overline\Omega$. 
On the boundary \eqref{eq:ball-inclusion} gives
\[
 f^{k+1}(\gamma)\subset B_{n+k+1}(t)
 \subset D(z_{n+k+1},\eta),
\]
and since the boundary image is compact, the maximum principle gives
\[
 \max_{w\in\overline\Omega}|f^{k+1}(w)-z_{n+k+1}|
 \leq\max_{w\in\gamma}|f^{k+1}(w)-z_{n+k+1}|<\eta,
\]
which proves the induction.

By \eqref{eq:filled-orbit}, every iterate is defined and holomorphic on
$\Omega$, and
\[
 |f^k(w)|\leq M+\eta\qquad(w\in\Omega,\ k\geq0).
\]
Thus $\Omega\subset F(f)$ by normality. 
Since the domain $\Omega$ is connected and meets $U_n$, we have $\Omega\subset U_n$.
\end{proof}

\begin{lemma}[Eventually embedded hyperbolic balls]\label{lem:embedded-balls}
For every fixed $t>0$, the restriction
\[
 \pi_n|_{r_t\D}\colon r_t\D\longrightarrow B_n(t),
 \qquad r_t=\tanh t,
\]
is a conformal isomorphism for all sufficiently large $n$. In particular,
these balls are simply connected and
\begin{equation}\label{eq:ball-area}
 \A_{\rho_{U_n}}(B_n(t))=\pi\sinh^2t
 \quad\text{for all sufficiently large }n.
\end{equation}
\end{lemma}

\begin{proof}
Take $n$ sufficiently large for the conclusion of Lemma~\ref{lem:local-filling} to hold.
Since $B_n(t)$ is a plane domain, its fundamental group is generated by classes represented by Jordan curves. By Lemma~\ref{lem:local-filling}, every Jordan curve in $B_n(t)$ bounds a simply connected region in $U_n$, and hence is null-homotopic in $U_n$. 

Since the inclusion $B_n(t) \hookrightarrow U_n$ induces the trivial homomorphism on fundamental groups, the lifting criterion for the universal cover $\pi_n$ gives a lift
\[
 s\colon B_n(t)\longrightarrow\D,\qquad s(z_n)=0,
 \qquad\pi_n\circ s=\iota,
\]
where $\iota \colon B_n(t) \to U_n$ is the inclusion.
The lift is holomorphic since $\pi_n$ is locally biholomorphic. 
On $r_t\D$, the map $s\circ\pi_n$ and the inclusion $\iota \colon r_t\D \to \D$ both fix $0$ and are both lifts of the same map with respect to $\pi_n$.
The uniqueness of lifts thus yields
\[
 s\circ\pi_n=\operatorname{id}_{r_t\D},
\]
hence $\pi_n$ is injective on $r_t\D$. Surjectivity onto the ball follows
from \eqref{eq:ball-uniformization}.

Using covering invariance and change of variables, we obtain
\begin{align*}
 \A_{\rho_{U_n}}(B_n(t))
 &=\int_{r_t\D}\rho_{\D}^2\dxdy
 =2\pi\int_0^{\tanh t}\frac{r}{(1-r^2)^2}\dd r\\
 &=\frac{\pi\tanh^2t}{1-\tanh^2t}
 =\pi\sinh^2t.
\end{align*}
\end{proof}

\begin{remark}\label{rem:injectivity-radius}
Equivalently, the injectivity radius of $(U_n,\rho_{U_n})$ at $z_n$ tends
to infinity. One can define this radius here as the supremum of the
numbers $t$ for which the centered universal cover is injective on
$(\tanh t)\D$. This does not assert that $U_n$ itself is simply connected.
The area in \eqref{eq:ball-area} uses the restriction of the ambient
density $\rho_{U_n}$, not the complete intrinsic density of $B_n(t)$.
\end{remark}

\begin{lemma}[Eventual injectivity on fixed-radius balls]\label{lem:injectivity}
For every fixed $t>0$, there is $N(t)$ such that $f|_{B_n(t)}$ is injective
for all $n\geq N(t)$.
\end{lemma}

\begin{proof}
Suppose otherwise. Choose $n_k\to\infty$ such that $f|_{B_{n_k}(t)}$ is
noninjective. Pass to a subsequence with $z_{n_k}\to a\in S$. The point $a$
is not a pole, and $z_{n_k+1}\to f(a)$.

Choose a proper local chart $f\colon A\to D$ about $a$ as in
Lemma~\ref{lem:local-charts}, with $D$ centered at $f(a)$. Its only possible
critical value is $f(a)$. By Lemma~\ref{lem:shrinking}, for large $k$,
\[
 B_{n_k}(t)\subset A,\qquad B_{n_k+1}(t)\subset D.
\]
By Lemma~\ref{lem:embedded-balls}, the target ball is simply connected for
large $k$. Moreover, the selected value $f(a)$ belongs to at most one of the
pairwise distinct components $U_{n_k+1}$. After discarding finitely many
indices, the target ball contains no critical value of the local chart.

Properness of the chart and avoidance of its critical values imply that
\[
 f\colon A\cap f^{-1}(B_{n_k+1}(t))
 \longrightarrow B_{n_k+1}(t)
\]
is an unramified covering. Each connected component of its source maps
conformally onto the simply connected target. The connected set
$B_{n_k}(t)$ lies in one such source component by
\eqref{eq:ball-inclusion}. Hence $f|_{B_{n_k}(t)}$ is injective, which is a
contradiction.
\end{proof}

\begin{proof}[Proof of Theorem~\ref{thm:mainB}]
The proof is similar to the proof of Theorem~\ref{thm:mainA}.

Suppose that some $z_0\in U$ has a bounded orbit. 
Let $K$ be the compact pole-free set from \eqref{eq:pole-free-K}.
Lemma~\ref{lem:deficit}, applied with $Q=P_1$, gives a constant $C<\infty$ such that
\begin{equation}\label{eq:uniform-deficit-K}
 \int_K\left[\rho_j^2-(f^*\rho_j)^2\right]_+\dxdy\leq C
 \qquad(j\geq1).
\end{equation}
Choose $t>0$ so that
\begin{equation}\label{eq:choose-t}
 \pi\sinh^2t>C.
\end{equation}
By Lemmas~\ref{lem:shrinking}, \ref{lem:embedded-balls}, and
\ref{lem:injectivity}, choose $N$ such that for every $n\geq N$,
\begin{equation}\label{eq:tail-properties}
 \begin{gathered}
 B_n(t)\subset K,\qquad f|_{B_n(t)}\text{ is injective},\\
 \A_{\rho_{U_n}}(B_n(t))=\pi\sinh^2t.
 \end{gathered}
\end{equation}
We also have $f(B_n(t))\subset B_{n+1}(t)$.

Fix $j$ and put $a_{j,n}:=\A_{\rho_j}(B_n(t))$. These quantities are finite
by \eqref{eq:finite-area-j}. Forward inclusion, injectivity, and change of
variables give, for $n\geq N$,
\begin{align}
 a_{j,n}-a_{j,n+1}
 &\leq \A_{\rho_j}(B_n(t))-\A_{\rho_j}(f(B_n(t)))\notag\\
 &=\int_{B_n(t)}\bigl(\rho_j^2-(f^*\rho_j)^2\bigr)\dxdy\notag\\
 &\leq\int_{B_n(t)}
 \left[\rho_j^2-(f^*\rho_j)^2\right]_+\dxdy.
 \label{eq:ball-one-step}
\end{align}
The balls lie in pairwise disjoint Fatou components and are contained in $K$.
Summing \eqref{eq:ball-one-step} from $n=N$ to $n=L > N$ and using
\eqref{eq:uniform-deficit-K} yields
\begin{equation}\label{eq:ball-telescoping}
 a_{j,N}-a_{j,L+1}\leq C.
\end{equation}
For this fixed $j$, disjointness and finite total area imply
\[
 \sum_{n\geq N}a_{j,n}
 \leq \A_{\rho_j}(\C\setminus P_j)<\infty.
\]
Hence $a_{j,L+1}\to0$ as $L \to \infty$. Letting $L\to\infty$ in
\eqref{eq:ball-telescoping} gives
\[
 \A_{\rho_j}(B_N(t))\leq C\qquad(j\geq1).
\]
Now let $j\to\infty$. By \eqref{eq:fatou-recovery} and monotone convergence,
\[
 \A_{\rho_{U_N}}(B_N(t))\leq C.
\]
This contradicts \eqref{eq:tail-properties} and \eqref{eq:choose-t}.
Therefore, no point of $U$ has a bounded orbit, proving
Theorem \ref{thm:mainB}.
\end{proof}



\begin{thebibliography}{BHKMT}

\bibitem[Bak76]{Bak76}
I.\,N. Baker,
\emph{An entire function which has wandering domains},
J. Austral. Math. Soc. Ser. A \textbf{22} (1976), no.~2, 173--176.
\href{https://doi.org/10.1017/S1446788700015287}{doi:10.1017/S1446788700015287}.

\bibitem[Bak]{Bak84}
I.\,N. Baker,
\emph{Wandering domains in the iteration of entire functions},
Proc. London Math. Soc. (3) \textbf{49} (1984), no.~3, 563--576.

\bibitem[BKL]{BKL90}
I.\,N. Baker, J. Kotus, and Y. L\"u,
\emph{Iterates of meromorphic functions. II. Examples of wandering domains},
J. London Math. Soc. (2) \textbf{42} (1990), no.~2, 267--278.
\href{https://doi.org/10.1112/jlms/s2-42.2.267}{doi:10.1112/jlms/s2-42.2.267}.

\bibitem[BBH]{BBH84}
K.\,F. Barth, D.\,A. Brannan, and W.\,K. Hayman,
\emph{Research problems in complex analysis},
Bull. London Math. Soc. \textbf{16} (1984), no.~5, 490--517.
\href{https://doi.org/10.1112/blms/16.5.490}{doi:10.1112/blms/16.5.490}.


\bibitem[BM]{BM}
A.\,F. Beardon and D. Minda,
\emph{The hyperbolic metric and geometric function theory},
in: S. Ponnusamy, T. Sugawa, and M. Vuorinen (eds.),
\emph{Quasiconformal Mappings and Their Applications},
Narosa, New Delhi, 2007, 9--56.

\bibitem[B]{Bergweiler}
W. Bergweiler,
\emph{Iteration of meromorphic functions},
Bull. Amer. Math. Soc. (N.S.) \textbf{29} (1993), no.~2, 151--188.
\href{https://doi.org/10.1090/S0273-0979-1993-00432-4}{doi:10.1090/S0273-0979-1993-00432-4}.

\bibitem[BHKMT]{BHKMT93}
W. Bergweiler, M. Haruta, H. Kriete, H.-G. Meier, and N. Terglane,
\emph{On the limit functions of iterates in wandering domains},
Ann. Acad. Sci. Fenn. Ser. A I Math. \textbf{18} (1993), no.~2, 369--375.
\href{https://www.acadsci.fi/mathematica/Vol18/bergweil.pdf}{Full text}.

\bibitem[BKL]{BKL92}
I. N. Baker, J. Kotus, and Y. L\"u,
\emph{Iterates of meromorphic functions. IV. Critically
finite functions},
Results Math. \textbf{22} (1992), nos.~3--4, 651--656.
\href{https://doi.org/10.1007/BF03323112}{doi:10.1007/BF03323112}.

\bibitem[BM]{BM}
W. Bergweiler and S. Morosawa,
\emph{Semihyperbolic entire functions},
Nonlinearity \textbf{15} (2002), no.~5, 1673--1684.
\href{https://doi.org/10.1088/0951-7715/15/5/316}{doi:10.1088/0951-7715/15/5/316}.

\bibitem[Bis]{Bis15}
C.\,J. Bishop,
\emph{Constructing entire functions by quasiconformal folding},
Acta Math. \textbf{214} (2015), no.~1, 1--60.
\href{https://www.math.stonybrook.edu/~bishop/papers/FoldingACTA.pdf}{Full text}.

\bibitem[BFJK]{BFJK}
K. Bara\'nski, N. Fagella, X. Jarque, and B. Karpi\'nska,
\emph{Fatou components and singularities of meromorphic functions},
Proc. Roy. Soc. Edinburgh Sect. A \textbf{150} (2020),
no.~2, 633--654.
\href{https://doi.org/10.1017/prm.2018.142}{doi:10.1017/prm.2018.142}.

\bibitem[BH]{BH89}
D.\,A. Brannan and W.\,K. Hayman,
\emph{Research problems in complex analysis},
Bull. London Math. Soc. \textbf{21} (1989), no.~1, 1--35.
\href{https://doi.org/10.1112/blms/21.1.1}{doi:10.1112/blms/21.1.1}.

\bibitem[EL84]{EL84}
A.\,E. Eremenko and M.\,Yu. Lyubich,
\emph{Iterates of entire functions},
Soviet Math. Dokl. \textbf{30} (1984), no.~3, 592--594;
translated from Dokl. Akad. Nauk SSSR \textbf{279} (1984), no.~1, 25--27.

\bibitem[EL87]{EL87}
A.\,E. Eremenko and M.\,Yu. Lyubich,
\emph{Examples of entire functions with pathological dynamics},
J. London Math. Soc. (2) \textbf{36} (1987), no.~3, 458--468.
\href{https://www.math.stonybrook.edu/~bishop/classes/math627.S13/Eremenko-Lyubich-Pathological.pdf}{Full text}.

\bibitem[EL92]{EL92}
A.\,E. Eremenko and M.\,Yu. Lyubich,
\emph{Dynamical properties of some classes of entire functions},
Ann. Inst. Fourier (Grenoble) \textbf{42} (1992), no.~4, 989--1020.
\href{https://doi.org/10.5802/aif.1318}{doi:10.5802/aif.1318}.


\bibitem[FJL]{FJL19}
N. Fagella, X. Jarque, and K. Lazebnik,
\emph{Univalent wandering domains in the Eremenko--Lyubich class},
J. Anal. Math. \textbf{139} (2019), no.~1, 369--395.
\href{https://arxiv.org/abs/1711.10629}{arXiv:1711.10629}.

\bibitem[FP]{FP26}
N. Fagella and L. Pardo-Sim\'on,
\emph{From pathological to paradigmatic: A retrospective on Eremenko and Lyubich's entire functions},
J. London Math. Soc. (2) \textbf{113} (2026), no.~1, e70382.
\href{https://doi.org/10.1112/jlms.70382}{doi:10.1112/jlms.70382}.

\bibitem[GK]{GK86}
L.\,R. Goldberg and L. Keen,
\emph{A finiteness theorem for a dynamical class of entire functions},
Ergodic Theory Dynam. Systems \textbf{6} (1986), no.~2, 183--192.
\href{https://doi.org/10.1017/S0143385700003394}{doi:10.1017/S0143385700003394}.

\bibitem[HL]{HL19}
W.\,K. Hayman and E.\,F. Lingham,
\emph{Research Problems in Function Theory},
Fiftieth Anniversary Edition, Problem Books in Mathematics,
Springer, Cham, 2019.
\href{https://doi.org/10.1007/978-3-030-25165-9}{doi:10.1007/978-3-030-25165-9}.

\bibitem[Her]{Her84}
M.\,R. Herman,
\emph{Exemples de fractions rationnelles ayant une orbite dense sur la sph\`ere de Riemann},
Bull. Soc. Math. France \textbf{112} (1984), no.~1, 93--142.
\href{https://doi.org/10.24033/bsmf.2002}{doi:10.24033/bsmf.2002}.

\bibitem[MS]{MS20}
D. Mart\'i-Pete and M. Shishikura,
\emph{Wandering domains for entire functions of finite order in the Eremenko--Lyubich class},
Proc. London Math. Soc. (3) \textbf{120} (2020), no.~2, 155--191.
\href{https://doi.org/10.1112/plms.12288}{doi:10.1112/plms.12288}.

\bibitem[McM]{McM}
C.\,T. McMullen,
\emph{Renormalization and 3-Manifolds Which Fiber over the Circle},
Annals of Mathematics Studies, vol.~142,
Princeton University Press, Princeton, NJ, 1996.

\bibitem[MRG]{MRG13}
H. Mihaljevi\'c-Brandt and L. Rempe-Gillen,
\emph{Absence of wandering domains for some real entire functions with bounded singular sets},
Math. Ann. \textbf{357} (2013), no.~4, 1577--1604.
\href{https://doi.org/10.1007/s00208-013-0936-z}{doi:10.1007/s00208-013-0936-z}.

\bibitem[Nic]{Nic13}
D.\,A. Nicks,
\emph{Wandering domains in quasiregular dynamics},
Proc. Amer. Math. Soc. \textbf{141} (2013), no.~4, 1385--1392.
\href{https://arxiv.org/abs/1101.1483}{arXiv:1101.1483}.

\bibitem[PRW]{PRW}
N. Prochorov, L. Rempe, and J. Waterman,
\emph{Absence of bounded-orbit wandering domains},
preprint, 2026.
\href{https://arxiv.org/abs/2609.28279v1}{arXiv:2609.28279v1}.


\bibitem[RvS]{RvS15}
L. Rempe-Gillen and S. van Strien,
\emph{Density of hyperbolicity for classes of real transcendental
entire functions and circle maps},
Duke Math. J. \textbf{164} (2015), no.~6, 1079--1137.
\href{https://doi.org/10.1215/00127094-2885764}{doi:10.1215/00127094-2885764}.

\bibitem[Sul]{Sul85}
D. Sullivan,
\emph{Quasiconformal homeomorphisms and dynamics. I. Solution of the Fatou--Julia problem on wandering domains},
Ann. of Math. (2) \textbf{122} (1985), 401--418.
\href{https://doi.org/10.2307/1971308}{doi:10.2307/1971308}.

\bibitem[Y]{Ye}
Z. Ye,
\emph{A new proof of no-wandering domain theorems without quasiconformal techniques},
preprint, 2026, version~1, 20 September 2026.
\href{https://arxiv.org/abs/2609.23834v1}{arXiv:2609.23834v1}.

\bibitem[Zhe]{Zhe03}
J.-H. Zheng,
\emph{Singularities and limit functions in iteration of meromorphic functions},
J. London Math. Soc. (2) \textbf{67} (2003), no.~1, 195--207.
\href{https://doi.org/10.1112/S0024610702003800}{doi:10.1112/S0024610702003800}.


\end{thebibliography}
\end{document}